\documentclass{amsart}

\usepackage{color}
\newtheorem{thm}{Theorem}
\newtheorem{lem}[thm]{Lemma}

\newtheorem{fact}[thm]{Fact}

\begin{document}

\title[Corrigendum to "The optimal range of the Calderon operator."]{Corrigendum to the paper "The optimal range of the Calderon operator and its applications." [J. Funct. Anal. 277 (2019), no. 10, 3513--3559.]}

\author[]{$^{1}$F. Sukochev, $^{2}$K. Tulenov, and $^{3}$D. Zanin}
\address{$^{1,3}$School of Mathematics and Statistics, University of New South Wales, Kensington,  2052, Australia}

\address{$^{2}$Al-Farabi Kazakh National University and Institute of Mathematics and Mathematical Modeling, Almaty, Kazakhstan}
\email{$^{1}$f.sukochev@unsw.edu.au, $^{2}$tulenov@math.kz, $^{3}$d.zanin@unsw.edu.au}

\begin{abstract} We fix a gap in the proof of a result in our earlier paper \cite{STZ}.
\end{abstract}

\maketitle

In this note we rectify a mistake contained in Lemma 22 in \cite[Section 5]{STZ}.
 Because of a mistake in the proof of \cite[Lemma 22]{STZ}, the main results of such section \cite[Theorem 21]{STZ} is not proved in full generality. We provide an alternative proof of \cite[Lemma 22]{STZ} and show that \cite[Theorem 21]{STZ}  still holds.

\section{Preliminaries}

For convenience, we keep almost all notations as in \cite{STZ}. So, for undefined notations and notions below, we refer the reader to \cite{STZ}.

Define the triangular truncation (with respect to a continuous chain) as usual: if the operator $A$ is an integral operator on the Hilbert space $L_2(-\pi,\pi),$ with the integral kernel $K,$ i.e.
$(Af)(t)=\int_{-\pi}^{\pi}K(t,s)f(s)ds, \quad t\in(-\pi,\pi),\quad f\in L_2(-\pi,\pi),$
then $T(A)$ is an integral operator with truncated integral kernel
$$((T(A))f)(t)=\int_{-\pi}^{\pi} {\rm sgn}(t-s) K(t,s)f(s)ds,\quad t\in(-\pi,\pi),\quad f\in L_2(-\pi,\pi).$$

It is established in \cite{STZ} that the operator $T$ acts from the ideal $\Lambda_{\log}(B(L_2(-\pi,\pi)))$ into $B(L_2(-\pi,\pi)),$ where $B(L_2(-\pi,\pi))$ is the algebra of all linear bounded operators on $L_2(-\pi,\pi).$  It is also claimed there that $\Lambda_{\log}(B(L_2(-\pi,\pi)))$ is the maximal domain of $T.$ Unfortunately, the proof of this claim contains a substantial oversight: the statement in \cite[Lemma 22]{STZ} is actually false. Nevertheless, we are able to fix this mistake and ensure that \cite[Theorem 21]{STZ} remains true as stated.

Recall the following auxiliary operators from \cite{STZ}. If $x \in \Lambda_{{\rm log}}(\mathbb{Z})$, then the {\it discrete Hilbert-type transform} $\mathcal{H}_d$ is defined by the formula
\begin{equation}\label{H dis}
(\mathcal{H}_dx)(n)=\frac1{\pi}\sum_{\substack{m\in\mathbb{Z}\\ m-n=1{\rm mod}2}}\frac{x(m)}{m-n},\quad n\in\mathbb{Z}.
\end{equation}
Also, define the {\it discrete version of the Calder\'{o}n operator} $S_d$ as
\begin{equation}\label{S dis}
S_dx=C_dx+C_d^{\ast}x=\frac{1}{n+1}\sum_{k=0}^{n}x(k)+\sum_{k=n}^{\infty}\frac{x(k)}{k+1},\quad x \in \Lambda_{{\rm log}}(\mathbb{Z}_{+}).
\end{equation}

\section{Lower distributional estimate for triangular truncation}\label{tt app}


The following result replaces Theorem 21 in \cite{STZ}.
\begin{thm}\label{tt lower thm} For every $x\in\Lambda_{{\rm log}}(\mathbb{Z}_+),$ there exists an operator $a\in B(L_2(-\pi,\pi))$ such that $\mu(a)=\mu(x)$ and
$$\mu(T(a))\geq \frac1{8\pi}S_d\mu(x).$$
\end{thm}

The proof of Theorem \ref{tt lower thm} relies on Lemma \ref{stz mistake lemma} and Lemma \ref{stz fix lemma} below.
The following fact is an easy exercis.

\begin{fact}\label{stz compute fact} If $x\in \ell_2(\mathbb{Z}),$ then
$$2i\widehat{\mathcal{H}_dx}(t)={\rm sgn}(t)\cdot\hat{x}(t),\quad t\in(-\pi,\pi),$$
where
$$\hat{x}(t)=\sum_{m\in\mathbb{Z}}x(m)e^{imt},\quad t\in(-\pi,\pi).$$
\end{fact}
\begin{proof} For every $k\in\mathbb{Z},$ we have
$$\int_{-\pi}^{\pi}{\rm sgn}(t)\cdot e^{ikt}dt=\int_0^{\pi}e^{ikt}dt-\int_{-\pi}^0e^{ikt}dt=
\begin{cases}
0,& k\mbox{ is even}\\
\frac{4i}{k},& k\mbox{ is odd}
\end{cases}.
$$
Thus, for every $n\in\mathbb{Z},$ we have
\begin{eqnarray*}\begin{split}\frac1{2\pi}\int_{-\pi}^{\pi}{\rm sgn}(t)\cdot\hat{x}(t)e^{-int}dt&=\frac1{2\pi}\sum_{m\in\mathbb{Z}}x(m)\int_{-\pi}^{\pi}{\rm sgn}(t)\cdot e^{i(m-n)t}dt\\
&=\frac{2i}{\pi}\sum_{\substack{m\in\mathbb{Z}\\ m-n=1{\rm mod}2}}\frac{x(m)}{m-n}\stackrel{\eqref{H dis}}=2i(\mathcal{H}_dx)(n).
\end{split}\end{eqnarray*}
\end{proof}

In the next lemma, we consider the von Neumann algebra $B(L_2(-\pi,\pi))$ and identify it with $B(L_2(\mathbb{T})).$ The purpose of the lemma below is to identify the mistake in \cite[Lemma 22]{STZ} (this mistake would be alleviated in Lemma \ref{stz fix lemma} below).

Denote by $\mathcal{D}$ the differential operator $\mathcal{D}:=\frac 1{i}\frac{d}{dt}$ defined on the Sobolev space $W^{1,2}(\mathbb{T}).$

\begin{lem}\label{stz mistake lemma} If $x\in\Lambda_{{\rm log}}(\mathbb{Z}),$ $a=x(\mathcal{D}),$ then
$$T(a)=2ip(\mathcal{H}_dx)(\mathcal{D})p+2iq(\mathcal{H}_dx)(\mathcal{D})q+px(\mathcal{D})q-qx(\mathcal{D})p,$$
where $p=M_{\chi_{(0,\pi)}}$ and $q=M_{\chi_{(-\pi,0)}}$ are multiplication operators on $L_2(\mathbb{T}).$
\end{lem}
\begin{proof} Suppose first that $x\in \ell_2(\mathbb{Z})$ and, therefore, $\hat{x}\in L_2(\mathbb{T}).$ One can write $a=x(\mathcal{D})$ as an integral operator of convolution type
$$(af)(t)=\frac1{2\pi}\int_{\mathbb{T}}\hat{x}(t-s)f(s)ds.$$

By the definition of $p$ and $q,$ we have
$$(papf)(t)=\frac1{2\pi}\int_{\mathbb{T}}\chi_{(0,\pi)}(t)\chi_{(0,\pi)}(s)\hat{x}(t-s)f(s)ds$$
and
$$(qaqf)(t)=\frac1{2\pi}\int_{\mathbb{T}}\chi_{(-\pi,0)}(t)\chi_{(-\pi,0)}(s)\hat{x}(t-s)f(s)ds.$$
If $t,s\in(0,\pi)$ or $t,s\in(-\pi,0),$ then $t-s\in(-\pi,\pi).$ Thus, by Fact \ref{stz compute fact},
$${\rm sgn}(t-s)\hat{x}(t-s)=2i\widehat{\mathcal{H}_dx}(t-s)$$
whenever $t,s\in(0,\pi)$ or $t,s\in(-\pi,0).$ By the definition of triangular truncation operator, we have
$$(T(pap)f)(t)=\frac{2i}{2\pi}\int_{\mathbb{T}}\chi_{(0,\pi)}(t)\chi_{(0,\pi)}(s)\widehat{\mathcal{H}_dx}(t-s)f(s)ds$$
and
$$(T(qaq)f)(t)=\frac{2i}{2\pi}\int_{\mathbb{T}}\chi_{(-\pi,0)}(t)\chi_{(-\pi,0)}(s)\widehat{\mathcal{H}_dx}(t-s)f(s)ds.$$
In other words,
\begin{equation}\label{stz meq1}
T(pap)=2ip(\mathcal{H}_dx)(\mathcal{D})p,\quad T(qaq)=2iq(\mathcal{H}_dx)(\mathcal{D})q.
\end{equation}

Similarly,
$$(paqf)(t)=\frac1{2\pi}\int_{\mathbb{T}}\chi_{(0,\pi)}(t)\chi_{(-\pi,0)}(s)\hat{x}(t-s)f(s)ds$$
and
$$(qapf)(t)=\frac1{2\pi}\int_{\mathbb{T}}\chi_{(-\pi,0)}(t)\chi_{(0,\pi)}(s)\hat{x}(t-s)f(s)ds.$$
Obviously, we have
$${\rm sgn}(t-s)=1,\quad t\in(0,\pi),\quad s\in(-\pi,0)$$
and
$${\rm sgn}(t-s)=-1,\quad t\in(-\pi,0),\quad s\in(0,\pi).$$
By the definition of triangular truncation operator, we get
$$(T(paq)f)(t)=\frac1{2\pi}\int_{\mathbb{T}}\chi_{(0,\pi)}(t)\chi_{(-\pi,0)}(s)\hat{x}(t-s)f(s)ds$$
and
$$(T(qap)f)(t)=-\frac1{2\pi}\int_{\mathbb{T}}\chi_{(-\pi,0)}(t)\chi_{(0,\pi)}(s)\hat{x}(t-s)f(s)ds.$$
In other words,
\begin{equation}\label{stz meq2}
T(paq)=px(\mathcal{D})q,\quad T(qap)=-qx(\mathcal{D})p.
\end{equation}

Combining \eqref{stz meq1} and \eqref{stz meq2}, we obtain
\begin{eqnarray}\begin{split}\label{stz meq3}
T(a)&=T(pap)+T(qaq)+T(paq)+T(qap)\\
&=2ip(\mathcal{H}_dx)(\mathcal{D})p+2iq(\mathcal{H}_dx)(\mathcal{D})q+px(\mathcal{D})q-qx(\mathcal{D})p.
\end{split}\end{eqnarray}
This proves the assertion for $x\in \ell_2(\mathbb{Z}).$ 

Define the mappings $L_1,L_2:\Lambda_{{\rm log}}(\mathbb{Z})\to B(L_2(-\pi,\pi))$ by setting
$$L_1:x\to T(x(\mathcal{D})),\quad x\in\Lambda_{{\rm log}}(\mathbb{Z})$$
and
$$L_2:x\to 2ip(\mathcal{H}_dx)(\mathcal{D})p+2iq(\mathcal{H}_dx)(\mathcal{D})q+px(\mathcal{D})q-qx(\mathcal{D})p,\quad x\in\Lambda_{{\rm log}}(\mathbb{Z}).$$
By Theorems 11 and 14 in \cite{STZ}, the operator $T:\Lambda_{{\rm log}}(B(L_2(-\pi,\pi)))\to B(L_2(-\pi,\pi))$ is bounded. Obviously,
$$\|L_1x\|_{\infty}\leq\|T\|_{\Lambda_{{\rm log}}(B(L_2(-\pi,\pi)))\to B(L_2(-\pi,\pi))}\|x\|_{\Lambda_{{\rm log}}},\quad x\in\Lambda_{{\rm log}}(\mathbb{Z}).$$
Similarly, by the $\ell_1\to \ell_{1,\infty}$ estimate for $\mathcal{H}_d$ and Theorem 14 in \cite{STZ} we have that $\mathcal{H}_d:\Lambda_{{\rm log}}(\mathbb{Z})\to \ell_{\infty}(\mathbb{Z})$ is a bounded operator. Obviously,
$$\|L_2x\|_{\infty}\leq 2\|\mathcal{H}_dx\|_{\infty}+2\|x\|_{\infty}\leq 2(1+\|\mathcal{H}_d\|_{\Lambda_{{\rm log}}\to\ell_{\infty}})\|x\|_{\Lambda_{{\rm log}}},\quad x\in\Lambda_{{\rm log}}(\mathbb{Z}).$$
By \eqref{stz meq3}, we have $L_1=L_2$ on $\ell_2(\mathbb{Z}).$ Since $\ell_2(\mathbb{Z})$ is dense in $\Lambda_{{\rm log}}(\mathbb{Z})$ and since both $L_1$ and $L_2$ are bounded, it follows that $L_1=L_2.$
\end{proof}

The following result replaces \cite[Lemma 22]{STZ}.
\begin{lem}\label{stz fix lemma} If $x\in\Lambda_{{\rm log}}(\mathbb{Z})$ is such that $x|_{2\mathbb{Z}}=0,$ then
$$\mu\big(p(\mathcal{H}_dx)(\mathcal{D})p\big)=\frac12\mu(\mathcal{H}_dx),$$
where $p=M_{\chi_{(0,\pi)}}$ is a multiplication operator.
\end{lem}
\begin{proof} The following observation is crucial: $\mathcal{H}_dx|_{2\mathbb{Z}+1}=0$ (it follows from the definition of $\mathcal{H}_d$ and the assumption $x|_{2\mathbb{Z}}=0$). Equivalently, the distribution $\widehat{\mathcal{H}_dx}$ is $\pi$-periodic. Now define an orthonormal basis $\{f_n\}_{n\in\mathbb{Z}}$ in the Hilbert space $L_2(0,\pi)$ by setting $f_n(t)=\pi^{-\frac12}e^{2int},$ $t\in(0,\pi).$ We claim that $\{f_n\}_{n\in\mathbb{Z}}$ is an eigenbasis for the operator $p(\mathcal{H}_dx)(\mathcal{D})p.$ By linearity and continuity, it suffices to prove this for the case when $x$ is a finitely supported sequence. In this case, $x\in \ell_2(\mathbb{Z})$ and, therefore, $\widehat{\mathcal{H}_dx}$ is a $\pi$-periodic function. Hence, for any $t\in(0,\pi),$ we have
\begin{align*}
(p(\mathcal{H}_dx)(\mathcal{D})p)f_n(t)&=\frac1{2\pi}\int_0^{\pi}\widehat{\mathcal{H}_dx}(t-s)f_n(s)ds=\frac1{2\pi^{\frac12}}f_n(t)\cdot\int_0^{\pi}\widehat{\mathcal{H}_dx}(-s)f_n(s)ds\\
&=\frac1{2\pi}f_n(t)\cdot\sum_{m\in\mathbb{Z}}(\mathcal{H}_dx)(2m)\int_0^{\pi}e^{2i(n-m)s}ds=\frac12(\mathcal{H}_dx)(2n)f_n(t).
\end{align*}
which proves the claim. The assertion follows from the claim.
\end{proof}

\begin{proof}[Proof of Theorem \ref{tt lower thm}] As the statement of the theorem depends only on $\mu(x),$ we may assume without loss of generality that
$$x(n)=
\begin{cases}
\mu(\frac{n-1}{2},x),& n\geq1\mbox{ is odd}\\
0,& n\mbox{ is even}\\
0,& n\leq 0
\end{cases}$$
and let $a=x(\mathcal{D}).$ Obviously, $x|_{2\mathbb{Z}}=0$ and $\mu(a)=\mu(x).$

On the other hand, it follows from Lemma \ref{stz mistake lemma} and from the equalities $q=1-p$ that
$$p\cdot T(a)\cdot p=2ip(\mathcal{H}_dx)p.$$
Thus, by \cite[Formulas (2.2) and (2.3), p. 27]{GK} we have
\begin{equation}\label{ttlt eq1}
\mu(T(a))\geq\mu(p\cdot T(a)\cdot p)\stackrel{L.\ref{stz mistake lemma}}{=}2\mu(p(\mathcal{H}_dx)(\mathcal{D})p)\stackrel{L.\ref{stz fix lemma}}{=}\mu(\mathcal{H}_dx).
\end{equation}
Obviously,
$$\frac1{2m+2n+1}\geq\frac1{2(m+n+2)}\geq\frac14\min\left\{\frac1{n+1},\frac1{m+1}\right\},\quad m,n\geq 0.$$
Then,
$$(\mathcal{H}_dx)(-2n)=\frac1{\pi}\sum_{m\geq0}\frac{\mu(m,x)}{2m+2n+1}\geq\frac1{4\pi}\sum_{m\geq0}\mu(m,x)\min\left\{\frac1{n+1},\frac1{m+1}\right\}, n\in\mathbb{Z}_+.$$

Thus,
$$(\mathcal{H}_dx)(-2n)\geq \frac1{8\pi}(S_d\mu(x))(n),\quad n\geq0.$$
It is now immediate that
\begin{equation}\label{ttlt eq2}
\mu(\mathcal{H}_dx)\geq\frac1{8\pi}S_d\mu(x).
\end{equation}
Combining \eqref{ttlt eq1} and \eqref{ttlt eq2}, we complete the proof.	
\end{proof}


\begin{thebibliography}{99}
\bibitem{DFWW} Dykema K., Figiel T., Weiss G., Wodzicki M. {\it Commutator structure of operator ideals.} Adv. Math. {\bf 185} (2004), no. 1, 1--79.
\bibitem{GK} I.C. Gohberg and M.G. Kre\v{i}n, {\it Introduction to the theory of linear non-selfadjoint operators}, Transl. Math. Monogr., {\bf 18}, Amer. Math. Soc., Providence, R.I., 1969.
\bibitem{STZ} Sukochev F., Tulenov K., Zanin D. {\it The optimal range of the Calder\'{o}n operator and its applications.} J. Funct. Anal. {\bf 277} (2019), no. 10, 3513--3559.
\end{thebibliography}
\end{document}